\newcommand{\dataversione}{\today}

\documentclass[11pt,reqno,a4paper]{amsart}

\usepackage{mathrsfs, graphics}
\usepackage{amssymb}
\usepackage[notref,notcite,final]{showkeys}
\usepackage{hyperref}\hypersetup{colorlinks=true, citecolor=blue}
\usepackage{graphicx, epstopdf}\graphicspath{{./immaginiprova/}}

\numberwithin{equation}{section}

\newtheoremstyle{mytheorem}
{}
{}
{\it}
{\parindent}
{\bf}
{.}
{ }
{\thmnumber{#2.~}\thmname{#1}\thmnote{~\rm#3}}

\newtheoremstyle{myremark}
{}
{}
{\rm}
{\parindent}
{\bf}
{.}
{ }
{\thmnumber{#2.~}\thmname{#1}\thmnote{~\rm#3}}

\newtheoremstyle{myparagraph}
{}
{}
{\rm}
{\parindent}
{\bf}
{.}
{ }
{\thmnumber{#2.~}\thmname{#1}\thmnote{#3}}

\theoremstyle{mytheorem}

\newtheorem{theorem}[subsection]{Theorem}
\newtheorem{lemma}[subsection]{Lemma}

\newtheorem{proposition}[subsection]{Proposition}

\theoremstyle{myremark}

\theoremstyle{myparagraph}

\newtheorem*{parag*}{}

\makeatletter
\def\@secnumfont{\sc}
\def\section{\@startsection{section}{1}%
\z@{1.5\linespacing\@plus .2\linespacing}{.7\linespacing}%
{\normalfont\sc\centering}}
\makeatother

\makeatletter
\def\ps@headings{\ps@empty
 \def\@evenhead{%
  \setTrue{runhead}%
  \normalfont\footnotesize
  \rlap{\thepage}\hfil
  \def\thanks{\protect\thanks@warning}%
  \leftmark{}{}\hfil}%
 \def\@oddhead{%
  \setTrue{runhead}%
  \normalfont\footnotesize\hfil
  \def\thanks{\protect\thanks@warning}%
  \rightmark{}{}\hfil \llap{\thepage}}%
\let\@mkboth\markboth}
\makeatother
\makeatletter
\renewenvironment{proof}[1][\proofname]{\par
  \pushQED{\qed}%
  \normalfont \topsep6\p@\@plus6\p@\relax
  \trivlist
  \itemindent\normalparindent
  \item[\hskip\labelsep
    \bfseries
    #1\@addpunct{.}]\ignorespaces
}{%
  \popQED\endtrivlist\@endpefalse
}
\providecommand{\proofname}{Proof}
\makeatother
\newcommand{\R}{\mathbb{R}}
\newcommand{\N}{\mathbb{N}}

\newcommand{\Leb}{\mathscr{L}}

\newcommand{\dV}{d_V\kern-1pt}

\newcommand{\trait}[3]{\vrule width #1ex height #2ex depth #3ex}

\newcommand{\trace}{\mathchoice%
  {\mathbin{\trait{.12}{1.2}{.03}\trait{.8}{0.09}{0.03}}}
  {\mathbin{\trait{.12}{1.2}{.03}\trait{.8}{0.09}{0.03}}}
  {\mathbin{\hskip.15ex\trait{.09}{.84}{0.02}\trait{.56}{.07}{.02}}\hskip.15ex}
  {\mathbin{\trait{.07}{.6}{.01}\trait{.4}{.06}{.01}}}}

\makeatletter
\@addtoreset{footnote}{section}
\makeatother

\begin{document}

	%
\pagestyle{empty}
\pagestyle{myheadings}
\markboth%
{\underline{\centerline{\hfill\footnotesize%
\textsc{Andrea Marchese}%
\vphantom{,}\hfill}}}%
{\underline{\centerline{\hfill\footnotesize%
\textsc{Lusin theorems for Radon measures}%
\vphantom{,}\hfill}}}

	%
\thispagestyle{empty}

~\vskip -1.1 cm

	%
{\footnotesize\noindent
[version:~\dataversione]%
\hfill
}

\vspace{1.7 cm}

	%
{\large\bf\centering
Lusin type theorems for Radon measures\\
}

\vspace{.6 cm}

	%
\centerline{\sc Andrea Marchese}

\vspace{.8 cm}

{\rightskip 1 cm
\leftskip 1 cm
\parindent 0 pt
\footnotesize

	%
{\sc Abstract.}
We add to the literature the following observation. If $\mu$ is a singular measure on $\mathbb{R}^n$ which assigns measure zero to every porous set and $f:\mathbb{R}^n\rightarrow\mathbb{R}$ is a Lipschitz function which is non-differentiable $\mu$-a.e., then for every $C^1$ function $g:\mathbb{R}^n\rightarrow\mathbb{R}$ it holds $$\mu\{x\in\mathbb{R}^n: f(x)=g(x)\}=0.$$ In other words the Lusin type
approximation property of Lipschitz functions with $C^1$ functions does not hold with respect to a general Radon measure.



\par
}

%
%

\section{Introduction}
\noindent It is well known (see Theorem 3.1.16 of \cite{Fe}) that given any Lipschitz function $f:\R^n\to\R$ and $\varepsilon>0$ there exists a $C^1$ function $g:\R^n\to\R$ such that
$$\Leb^n\{x\in\R^n:f(x)\neq g(x)\}<\varepsilon,$$
where $\Leb^n$ denotes the Lebesgue measure on $\R^n$. In this note we prove that in general it is not possible to replace the Lebesgue measure with any Radon measure $\mu$. Indeed the following theorem shows that such approximation is not available
whenever $\mu$ is a singular measure on $\R^n$ which assigns measure zero to every porous set (see \S\ref{s2} for the definition of porosity) and $f:\R^n\to\R$ is a Lipschitz function which is non-differentiable $\mu$-almost everywhere.\\
\begin{theorem}\label{main}
Let $\mu$ be a singular measure on $\R^n$ which assigns measure zero to every porous set. Let $f:\R^n\to\R$ be a Lipschitz function which is non-differentiable $\mu$-a.e. Then for every $C^1$ function $g$, it holds
$$\mu(\{x\in\R^n: f(x)=g(x)\})=0.$$
\end{theorem}
\noindent The validity of Lusin type approximation properties in metric measure spaces has recently attracted some attention. For example, in \cite{HM} and \cite{Sp} the validity of Lusin type theorems for horizontal curves in Carnot Groups is studied and \cite{Da}
extends a Lusin type theorem for gradients, originally established in \cite{Al}, to the framework of metric measure spaces with a differentiability structure. The forthcoming paper \cite{MS}, provides a deeper investigation on the latter problem in the special metric measure space given by the Euclidean space $\R^n$ endowed with an arbitrary Radon measure $\mu$, analyzing the possibility to prescribe not only the differential, but also some non-linear blowups at many points. The class of the ``admissible'' blowups is determined in terms of certain geometric properties of the measure $\mu$, namely in terms of the decomposability bundle of $\mu$, introduced in \cite{AM}. Finally, let us mention the paper \cite{Sh}, where a result in the spirit of \cite{Al} was proved for maps from an infinite dimensional locally convex space, endowed with a Gaussian measure, to its Cameron-Martin space.

\noindent This paper is organized as follows. In \S \ref{s2} we recall two facts which are necessary to guarantee that the content of Theorem \ref{main} is non-empty: firstly the existence of a singular measure $\mu$ on $\R^n$ which assigns measure zero to every porous set (Proposition \ref{doublingmeasure}) and secondly the existence of a Lipschitz function $f:\R\to\R$ which is non-differentiable $\mu$-almost everywhere (Proposition \ref{nondiff}). These results are already present in the literature. Nevertheless, for the reader's convenience and in order to keep this note self contained, we present here slightly simplified versions of the original proofs. In \S \ref{s3} we prove Theorem \ref{main}. In \S\ref{s4} we briefly discuss the possibility to extend and improve the result contained in \cite{Al}. We observe that, in the one-dimensional case, the result is valid replacing the Lebesgue measure with any Radon measure and we show that, except for atomic measures, it is not possible to find Lipschitz functions with prescribed non-linear first order approximation at many points.

\subsection*{Acknowledgements}
The author is indebted to Bernd Kirchheim and David Preiss for valuable discussions.

\section{Notations and Prerequisites}\label{s2}
\noindent All the sets and functions considered in this note are tacitly assumed to be Borel measurable and measures are defined on the Borel $\sigma$-algebra. Moreover measures are positive, locally finite and \emph{inner regular} (i.e. the measure of a set can be approximated from within by compact subsets). As usual, we say that a measure $\mu$ is \emph{absolutely continuous} with respect to a measure $\nu$ (and we write $\mu\ll\nu$) if $\mu(E)=0$ for every Borel set $E$ such that $\nu(E)=0$. We say that $\mu$ is \emph{supported} on a set $E$ if $\mu(E^c)=0$ and we say that $\mu$ is \emph{singular} with respect to $\nu$ if there exists a Borel set $E$ such that $\nu(E)=0$ and $\mu$ is supported on $E$. When words like ``nullset'' and ``singular measure'' are used without further specification, they implicitly refer to the Lebesgue measure.\\
\noindent We say that a set $E\subset\R^n$ is \emph{porous at a point} $x$ if there exists constant $C(x)>0$, a positive sequence $r_k\rightarrow 0$ and a sequence of points $y_k\in B(x,r_k)$ such that
\begin{equation}\label{eqpor}
B(y_k,C(x)r_k)\subset B(x,r_k)\;\;\; {\rm{and}}\;\;\;E\cap B(y_k,C(x)r_k)=\emptyset,
\end{equation}
where we denoted by $B(y,r)$ the open ball centered at $y$ with radius $r$.
We say that $E$ is \emph{porous} if it is porous at every point $x\in E$. To guarantee that Theorem \ref{main} actually applies to a non-trivial class of pairs $(f,\mu)$ we need to prove first of all the existence of a singular measure on $\R^n$ which assigns measure zero to every porous set. Clearly it is sufficient to prove it for $n=1$. A proof of this fact can be found in \cite{Tk}. For the reader's convenience we present a slightly simpler, self-contained proof.

\begin{proposition}\label{doublingmeasure}
There exists a singular measure $\nu$ on the real line such that $\nu(P)=0$ whenever $P$ is porous.
\end{proposition}

\noindent The construction uses the idea of Riesz product measures. On $[0,1]$ we call $i$-th generation of dyadic intervals, $i=0,1,2,\ldots$, all the intervals of the form
$$I=[a2^{-i},(a+1)2^{-i}],\; {\rm{for}}\; a=0,\ldots,2^{i}-1.$$
Given a measure $\mu$ and a measurable function $f$ we denote by $f\mu$ the measure satisfying
$$f\mu(A):=\int_Af\;d\mu,$$
for every Borel set $A$. In particular, we write $\mu\trace A$ to indicate the measure $f\mu$, where $f=1_A$ is the characteristic function of a Borel set $A$ assuming values 0 and 1.
Finally we say that $x\in\R$ is a \emph{Lebesgue continuity point} for the function $f$ with respect to the measure $\mu$ if we have
$$\frac{1}{\mu(B(x,r))}\int_{B(x,r)}|f(y)-f(x)|\;{\rm{d}}\mu\rightarrow 0, \;\;\; {\rm{as}}\;{r\rightarrow 0}.$$

\noindent We will make use of the following lemma, which is a well-known fact. Being unable to find a precise reference, we prefer to include its proof.
\begin{lemma}[(Martingale Theorem)]\label{martin}
Let $(\mu_i)_{i\in\N}$ be a sequence of probability measures on $[0,1]$. Assume that $\mu_i=f_n\Leb_1$, where $f_i$ is constant on the dyadic intervals of the $i$-th generation. Assume moreover that $\mu_j(I)=\mu_i(I)$ for every
dyadic interval $I$ of the $i$-th generation, for every $j>i$. Then $\mu_i$ weakly converges to a probability measure $\nu$. Moreover the Radon-Nikodym derivative $f$ of the absolutely continuous part of $\nu$ satisfies
$$f=\lim_{i\rightarrow\infty}f_i,\;\;\;\Leb_1-{\rm{a.e.}}$$
\end{lemma}

\begin{proof}
By the compactness theorem for measures (see Proposition 2.5 of \cite{DL}), there is a subsequence $\mu_{i_h}$ weakly converging to a measure $\nu$. Since the dyadic intervals generate the Borel $\sigma$-algebra, the hypotheses of the theorem guarantee that actually
the whole sequence $\mu_i$ converges to $\nu$. To prove the second part of the theorem, denote $\nu_s$ the singular part of $\nu$ and let $S\subset[0,1]$ be a nullset such that $\nu_s([0,1]\setminus S)=0$. Fix a point $x\in[0,1]\setminus S$ with the following properties:
\begin{itemize}
  \item $x$ is a point of Lebesgue continuity for $f$ with respect to $\Leb_1$;
  \item $x$ is a continuity point for every $f_n$;
  \item $2^i\nu_s(I_i)\to 0$ as $i\to\infty$,
\end{itemize}
where we denoted by $I_i$ the dyadic
interval of the $i$-th generation containing $x$ (the second property guarantees that such interval is unique). Notice that these three properties are satisfied at $\Leb_1$-almost every point in $[0,1]$ and in particular the third property follows from the Besicovitch Differentiation Theorem (see Theorem 2.10 of \cite{DL}). Observe that $\{I_i\}_{i\in\N}$ is a family of sets of \emph{bounded eccentricity}, i.e. there exists $C>0$ such that each $I_i$ is contained in a ball $B$, centered at $x$, with $\Leb_1(I_i) \geq C\Leb_1(B)$. Therefore the Lebesgue Theorem (see Theorem 7.10 in \cite{Ru1}) yields:
$$f_i(x)=\frac{\mu_i(I_i)}{\Leb_1(I_i)}=\frac{\mu(I_i)}{\Leb_1(I_i)}=\frac{\int_{I_i}f\;{\rm{d}}\Leb_1}{\Leb_1(I_i)}+\frac{\mu_s(I_i)}{\Leb_1(I_i)}\rightarrow f(x), \;\;\; {\rm{as}}\;{i\rightarrow\infty}.$$
\end{proof}

\noindent The proof of Proposition \ref{doublingmeasure} uses a blowup argument. Given a Radon measure $\nu$ on $\R$ and a point $x$ we define the measure $\nu_{x,r}$ by
$$\nu_{x,r}(A):=\nu(x+rA), \;{\rm{for\,every\;Borel\;set\;}}A.$$ 
We denote by ${\rm{Tan}}(\nu,x)$ the set of the \emph{blowups} of $\nu$ at $x$, i.e. all the possible limits of the form
$$\lim_{r_i\searrow 0}\kappa_i$$
where
\begin{equation}\label{eqbu}
\kappa_i:=\frac{\nu_{x,r_i}\trace B(0,1)}{\nu(B(x,r_i))}.
\end{equation}

\noindent The following lemma gives a sufficient condition for a measure to assign measure zero to every porous set.

\begin{lemma}\label{lemmaporous}
Let $\nu$ be a locally finite measure on the line, such that for $\nu$-a.e. $x$ and for every $\eta\in {\rm{Tan}}(\nu,x)$, it holds $\Leb_1\ll\eta$. Then $\nu(P)=0$ for every porous set $P\subset \R$.
\end{lemma}
\begin{proof}
We assume by contradiction that $\nu$ satisfies the hypotheses of the lemma but there exists a porous set $P$ with $\nu(P)>0$. It is a general fact about tangent measures (see Remark 3.13 of \cite{DL}) that if $E$ is a Borel set, then ${\rm{Tan}}(\nu\trace E,x)={\rm{Tan}}(\nu,x)$ for $\nu$-a.e $x\in E$. Then for $\nu$-a.e. $x\in P$, every blowup $\eta$ of $\nu\trace P$ at $x$ is an element of ${\rm{Tan}}(\nu,x)$. In particular, by hypothesis, $\eta(B)>0$ for every open ball $B\subset B(0,1)$. Instead we show that, for every $x\in P$, it is possible to find a blowup $\eta$ of $\nu\trace P$ at the point $x$ and a non-trivial ball $B\subset B(0,1)$ such that $\eta(B)=0$. Fix $x\in P$ and let $C:=C(x)$, $(r_k)_{k\in\N}$ and $(y_k)_{k\in\N}$ as in \eqref{eqpor}. 
Possibly passing to a subsequence, we may assume that $(y_k-x)/r_k$ converges to a point $y\in \overline{B(0,1-C)}$. This implies that, for every subsequence of $(r_k)_{k\in\N}$, such that the corresponding rescaled measures (defined in \eqref{eqbu}) converge weakly to a measure $\eta\in {\rm{Tan}}(\nu\trace P,x)$, it holds $\eta(B(y,C/2))=0$.
\end{proof}

\begin{proof}[Proof of Proposition \ref{doublingmeasure}]
Consider the 1-periodic function $\varphi:\R\rightarrow\R$ which agrees with $2\chi_{[0,1/2]}-1$ on $[0,1]$ and consider a non-increasing sequence of positive numbers $(a_i)$, $i=0,1,2,\ldots,$ such that $a_0<1, a_i\searrow 0$ and $\sum_{i}a_i^2=+\infty$. Further hypotheses on $(a_i)$ will be specified later. 
Define on $[0,1]$ the functions
$$\varphi_i(x)=a_i\varphi(2^ix),\;\;\Phi_N=\sum_{i=0}^N\varphi_i,\;\;\psi_i=1+\varphi_i,\;\;\Psi_N=\prod_{i=0}^N\psi_i.$$
Consider now the measures $\mu_N=\Psi_N \Leb_1$. By the Martingale Theorem there exists a measure $\nu$ such that $\mu_N\stackrel{\ast}{\rightharpoonup}\nu$ as $N\to\infty$ and moreover $\Psi_N\rightarrow\frac{d\nu_{abs}}{dx}$ (the Radon-Nikodym
derivative of the absolutely continuous part of $\nu$). We will prove that $\nu$ is a singular measure and, for a suitable choice of $(a_n)$ it satisfies $\nu(P)=0$, for every porous set $P$.\\

\noindent {\bf{$\nu$ is singular.}} To prove that $\nu$ is singular, according to Lemma \ref{martin} it is sufficient to prove that $\liminf_N \Psi_N=0, \Leb_1$-a.e. Notice now that for $|t|<1$ there holds
$$\log(1+t)\leq t-\frac{t^2}{8},$$
hence we have
$$\log(\Psi_N)=\sum_{i=1}^N\log(1+\varphi_i)\leq\sum_{i=1}^N\left(\varphi_i-\frac{\varphi_i^2}{8}\right)=\Phi_N-\sum_{i=1}^N\frac{a_i^2}{8}.$$
Since the random variable $\Phi_N$ has expected value $E(\Phi_N)=0$ and variance $\sigma^2(\Phi_N)=\sum_{i=1}^N a_i^2$, then Chebyshev inequality (see 5.10.7 of \cite{As}) gives
$$
\Leb_1\left(\left\{x\in[0,1]:\Phi_N(x)>\sum_{i=1}^N\frac{a_i^2}{16}\right\}\right)\leq\frac{16^2}{\sum_{i=1}^Na_i^2},
$$
and the right-hand side tends to zero as $N\rightarrow\infty$ because $\sum a_i^2=+\infty$. Therefore we have
$$\liminf_N\Psi_N=\exp\left(\liminf_N\left(\Phi_N-\frac{\sum_{i=1}^Na_i^2}{8}\right)\right)=0,\;\;\; \Leb_1-{\rm{a.e}}.$$

\noindent {\bf{$\nu(P)=0$ whenever $P$ is a porous set.}}
Now we make the choice $a_0=a_1=1/\sqrt{2}$ and for $i>1$ $a_i:=i^{-1/2}$; we want to show that for $\nu$-a.e. point $x\in(0,1)$, every blowup of $\nu$ at $x$ gives positive measure to every non-trivial interval $J\subset(-1,1)$. By Lemma \ref{lemmaporous}, this guarantees that every porous set is $\nu$-null.

\noindent Consider a point $x\in(0,1)$, a measure $\eta\in {\rm{Tan}}(\nu,x)$ and a sequence $r_j\searrow 0$ such that $\eta=\lim_j \kappa_j$, where $\kappa_j$ is defined according to \eqref{eqbu}. We may further assume 
$$r_j\leq {\rm{dist}}(x,B(0,1)^c).$$ For every $j\in\N$, there exist $i\in\N$, and a dyadic interval $I_i(x)$, of the $i$-th generation, containing $x$, such that it also contains $x+r_j$ or $x-r_j$, but no interval in the next generation has the same property. Note that $I_i(x)$ cannot contain both $x+r_j$ and $x-r_j$. In particular we have 
$$r_j\leq|I_i(x)|\leq 2r_j.$$ 
Denote by $I'_i(x)$ the adjacent dyadic interval of the same generation as $I_i(x)$, that together with $I_i(x)$ covers $(x-r_j,x+r_j)$. We claim that, eventually in $j$ (remember that the parameter $i$ depends on $j$ in a monotone way), the ratio $$c_j(x)=\frac{\mu_i(I_i(x))}{\mu_i(I'_i(x))}$$ satisfies
\begin{equation}\label{eclaim}
e^{-4}\leq c_j(x)\leq e^4,\;{\rm{for}}\;\nu {\rm{-a.e.}}\;x\in (0,1).
\end{equation}
This would be sufficient to prove that $\eta(J)>0$ for every non-trivial closed interval $J\subset(-1,1)$. Indeed we have
\begin{equation}\label{ebuz}
\nu(J)\geq\limsup_j\kappa_j(J)\geq\limsup_j\frac{\nu(I)}{\nu(I_i(x)\cup I'_i(x))},
\end{equation}
where $I$ is the largest dyadic interval contained in $x+r_jJ$. The fact that $\nu(I)=\mu_m(I)$ for every $m$ sufficiently large and \eqref{eclaim} imply that the ratio in \eqref{ebuz} is bounded from below by a positive constant.\\

\noindent To prove the claim \eqref{eclaim}, fix $x\in(0,1)$ and let $(\sigma_k(x))_{k\in\N}$ be the unique sequence made of 0's and 1's such that
\begin{equation}\label{defsigma}
\min\{I_i(x)\}=\sum_{k=0}^i 2^{-k}\sigma_k(x)
\end{equation}
(see Figure \ref{segmenti}), and analogously define $(\sigma'_k(x))_{k=1}^i$ replacing $I_i$ with $I'_i$ in \eqref{defsigma}.\\
\begin{figure}[htbp]
\begin{center}
\scalebox{1}{
\input{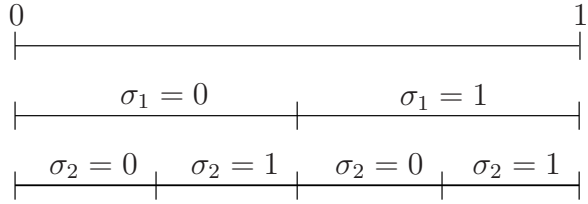}
}
\end{center}
\caption{Values of the function $\sigma_k(x)$, for $k=1,2$.}
\label{segmenti}
\end{figure}

\noindent Obviously we have
$$\max\{c_j(x),c_j(x)^{-1}\}\leq\prod_{k=k_0+1}^i\frac{1+a_k}{1-a_k},$$
where $k_0$ is the last index smaller than $i=i(j)$ such that $\sigma_{k_0}(x)=\sigma'_{k_0}(x)$. Notice that if $k_0<i-1$ and $I_i'(x)$ is the left neighborhood of $I_i(x)$, we have $\sigma_{k_0+1}(x)=1$ and $\sigma_k(x)=0$ for every $k=k_0+2,\ldots,i$; vice-versa if $I_i'(x)$ is the right neighborhood of $I_i(x)$, we have $\sigma_{k_0+1}(x)=0$ and $\sigma_k(x)=1$ for every $k=k_0+2,\ldots,i$.

\noindent For $\ell=0,1$, and for $n\geq 2$ denote
$$E^\ell_j=\{x\in(0,1):\sigma_k(x)=\ell,\;\;\; {\rm{for\;every\;}} k\in[i-i^{1/2}+2,i]\}.$$
Observe that, for $j$ sufficiently large, the set of points $x$ such that $c_j(x)\not\in[e^{-4},e^4]$ is contained in $E^0_j\cup E^1_j$. Indeed assume by contradiction that $x\not\in E^0_j\cup E^1_j$, but either $c_j(x)>e^4$ or $c_j(x)<e^{-4}$. In the first case we have
\begin{equation}\label{equazz}
\prod_{k=k_0+1}^i\frac{1+a_k}{1-a_k}>e^4.
\end{equation} 
Since $\log(1+t)<t$, we have
$$\prod_{k=k_0+1}^i\frac{1+k^{-1/2}}{1-k^{-1/2}}=\exp\left(\log\left(\prod_{k=k_0+1}^i\frac{1+k^{-1/2}}{1-k^{-1/2}}\right)\right)<
\exp\left(2\sum_{k=k_0+1}^ik^{-1/2}\right)$$
But it is easy to see that.
$$2\sum_{k=k_0+1}^ik^{-1/2}<4,$$ 
whenever $k_0>i-i^{1/2}$, hence \eqref{equazz} implies that $x\in E^0_j\cup E^1_j$. The second case can be proved analogously.\\

\noindent Eventually we compute
$$\nu(E^\ell_j)\leq\prod_{k=i-i^{1/2}+2}^i\frac{1+k^{-1/2}}{2}\leq2^{-i^{1/2}+2}\prod_{k=i-i^{1/2}}^i(1+k^{-1/2})\leq 2^{-i^{1/2}+4}.$$
Therefore
$$\nu\left(\bigcap_{h=2}^{\infty}\bigcup_{j=h}^{\infty}(E^0_j\cup E^1_j)\right)=0$$
and since, by the observation above, this set contains the set of points $x$ such that $c_j(x)\not\in[e^{-4},e^4]$ frequently, the claim \eqref{eclaim} is proved.
\end{proof}

\noindent A function $f:\R^n\to\R$ is called $L$-Lipschitz ($L>0$) if 
$$|f(y)-f(x)|\leq L|y-x|,\;\;\; {\rm{for\;every\;}}x,y\in\R^n.$$
We conclude this section proving that given a compact nullset $E\subset\R$, there exists a 1-Lipschitz function $f:\R\to\R$ which is non-differentiable at any point $x\in E$.\\ 

\noindent The original proof of a stronger statement is contained in \cite{Za} (in particular the proposition is valid even without the compactness hypothesis on the set $E$). The proof we present uses the Baire Theorem (see Theorem 2.2 of \cite{Ru2}).

\begin{proposition}\label{nondiff}
Given a compact nullset $E\subset\R$, there exists a 1-Lipschitz function $f:\R\to\R$ which is non-differentiable at all points $x\in E$.
\end{proposition}

\begin{proof}
Throughout the proof we will denote by $X$ the complete metric space of real valued 1-Lipschitz functions on the real line, endowed with the supremum distance. We will actually prove a stronger statement: namely that the family of 1-Lipschitz functions $f:\R\rightarrow\R$ such that $f$ is non-differentiable at the points of $E$ is \emph{residual} (i.e. it contains the intersection of countably many open dense sets), and in particular, by the Baire Theorem, it is dense in $X$.\\

\noindent Define inductively an infinitesimal sequence of positive numbers $(\varepsilon_i)_{i\in\N}$ and a sequence of open sets $(E_i)_{i\in\N}$, with the following properties
\begin{itemize}
  \item $E\subset E_{i+1}\subset E_i$;
  \item $E_i$ is a finite union of disjoint open intervals;
  \item $\Leb_1(E_i)\leq\varepsilon_i;$
  \item $\varepsilon_{i+1}\leq\alpha_i\varepsilon_i$, where 
  $$\alpha_i:=\min\{\Leb_1(I):I\; {\rm{is\;a\;connected\;component\;of}}\; E_i\}.$$
\end{itemize}
Define the following subsets of $X$ (we will write c.c. for ``connected component''):
$$U_i=\{g\in X:g(b)-g(a)>(b-a)-\varepsilon_{i+1},\; {\rm{whenever}}\; (a,b)\; {\rm{is\;a\;c.c.\;of}}\; E_i\},$$
$$V_i=\{g\in X:g(b)-g(a)<\varepsilon_{i+1}-(b-a),\; {\rm{whenever}}\; (a,b)\; {\rm{is\;a\;c.c.\;of}}\;  E_i\},$$
$$A_j=\bigcup_{i\geq j}U_i,\;\;\;B_j=\bigcup_{i\geq j}V_i.$$

\noindent Obviously $U_i$ and $V_i$ are open sets for every $i$, and therefore $A_j$ and $B_j$ are also open, for every $j$. Moreover, $U_i$ and $V_i$ are $2\varepsilon_i$-nets in $X$, by which we mean that for every element $\phi\in X$ there is an element $\phi_i\in U_i$ (respectively $V_i$) such that ${\rm{dist}}(\phi,\phi_i)\leq 2\varepsilon_{i}$. To prove this fact, consider for every function $\phi\in X$ the function
$$\phi_i(x)=\phi\bigg(x-\int_{-\infty}^x\chi_{E_i}(t)\;{\rm{d}}t\bigg)+\int_{-\infty}^x\chi_{E_i}(t)\;{\rm{d}}t,$$
which has the following properties: $\phi'_i(x)=\phi'(x)$ for a.e. $x\not\in E_i$ and $\phi'_i(x)=1$ if $x\in E_i$.
This is clearly an element of $U_i$ and $\|\phi-\phi_i\|_{\infty}\leq 2\varepsilon_i$. The proof that $V_i$ is a $2\varepsilon_i$-net is analogous.

\noindent As a consequence, $A_j$ and $B_j$ are dense for every $j$.
Finally,
$$A=\left(\bigcap_{j=1}^{\infty}A_j\right)\cap\left(\bigcap_{j=1}^{\infty}B_j\right)$$
is a residual set in $X$ (in particular it is dense).\\

\noindent Next we prove that every function $f\in A$ is not differentiable at any point of $E$. More precisely, we claim that
$$f'_+(x)=\limsup_{|h|\searrow 0}\frac{f(x+h)-f(x)}{h}=1$$
and
$$f'_-(x)=\liminf_{|h|\searrow 0}\frac{f(x+h)-f(x)}{h}=-1$$
for every $x\in E$.
Fix $\varepsilon>0$ and take $i\in\N$ such that $3\varepsilon_i<\varepsilon$, and $f\in U_i$. Let $I=(a,b)$ be the connected component of $E_i$ containing $x$. Take a point $y\in I$ such that
$${\rm{dist}}(x,y)\geq \frac{\Leb_1(I)}{3}.$$
Let $I'$ be the open interval with end points $x$ and $y$. Since on $(a,b)$ we have $f'\leq 1$ a.e. and $f(b)-f(a)\geq b-a-\varepsilon_{i+1}$, then we also have 
$$\int_{I'} f'(t)\;{\rm{d}}t\geq |x-y|-\varepsilon_{i+1}.$$ Therefore we conclude:
$$\frac{f(y)-f(x)}{y-x}\geq\frac{|y-x|-\varepsilon_{i+1}}{|y-x|}\geq 1-\frac{3\varepsilon_{i+1}}{\Leb_1(I)}\geq 1-\frac{3\varepsilon_{i+1}}{\alpha_i}\geq 1-3\varepsilon_i\geq1-\varepsilon.$$
Similarly we can prove that $f'_-(x)=-1$ for every $x\in E$.
\end{proof}

\section{Proof of Theorem \ref{main}}\label{s3}
\noindent By Proposition \ref{doublingmeasure} and Proposition \ref{nondiff} we deduce that the class of pairs $(\mu,f)$ satisfying the assumption of Theorem \ref{main} is non empty, at least for $n=1$. Indeed it is sufficient to consider the measure $\nu$ determined by Proposition \ref{doublingmeasure} and then to take any compact nullset $E$ with $\nu(E)>0$ and to define $\mu:=\nu\trace E$. Eventually, one can find the appropriate function $f$, applying Proposition \ref{nondiff} to the set $E$. To prove the same fact for $n>1$, one should replace our Proposition \ref{nondiff} with Theorem 1.13 of \cite{DPR}.\\

\noindent To prove Theorem \ref{main} assume by contradiction that there exist a $C^1$ function $g$ such that, denoting $$A:=\{x\in\R^n: g(x)=f(x)\},$$ there holds $\mu(A)>0$. We can assume that $f$ is 1-Lipschitz and $g$ is globally $L$-Lipschitz for some $L>0$.

\noindent Denote $h:=f-g$. Observe that $h$ is $(1+L)$-Lipschitz and $h=0$ on $A$. We claim that $Dh$ exists and it is equal to $0$ at $\mu$-a.e. point of $A$, which is a contradiction because it would imply that $f$ is differentiable on a set of positive measure $\mu$.\\

\noindent To prove the claim, consider the set $P\subset A$ of points where either $Dh$ does not exist or $Dh\neq 0$. In particular, for every $x\in P$, there exists a constant $C(x)>0$ and a sequence of points $y_k\to x$ such that
\begin{equation}
|h(y_k)|>C(x)|y_k-x|,
\end{equation}
for every $k\in\N$. Then for every $x\in P$ and for every $k\in\N$ we would have $h\neq 0$ on the open ball $B_k$ centered at $y_k$ with radius ${\frac{C(x) |y_k-x|}{L+1}}$.
Since $P\subset A$ this implies that the set $P$ is porous. Hence $\mu(P)=0$.

\section{Lusin type theorem for gradients}\label{s4}
\noindent In this section we discuss the possibility to extend and possibly to improve the result of \cite{Al}, when we replace the Lebesgue measure with any Radon measure. We will consider only the one-dimensional setting. The higher dimensional case and further results are discussed in \cite{MS} The first statement is that the result of \cite{Al} is valid with respect to any Radon measure.

\begin{theorem}\label{grad1}
Let $g:\R\to\R$ be a bounded Borel function and $\mu$ be a Radon measure on $\R$. Then for every $\varepsilon>0$ there exist a set $E\subset\R$ with $\mu(\R\setminus E)<\varepsilon$ and a $C^1$ function $f:\R\to\R$ such that $f'=g$ on the set $E$.
\end{theorem}
\begin{proof}
Fix $\varepsilon>0$. By the standard Lusin theorem (see Theorem 2.24 of \cite{Ru1}), there exist a set $E\subset\R$ with $\mu(\R\setminus E)<\varepsilon$ and a bounded and continuous function $h:\R\to\R$ such that $h=g$ on the set $E$. 
Denote 
$$f(x):=\int_0^xh(t) dt.$$ 
Clearly $f$ is $C^1$ and it holds $f'=h=g$ on the set $E$.
\end{proof}

\noindent In the next step we want to replace in theorem \ref{grad1} $C^1$ functions with Lipschitz functions. Clearly when $\mu$ is the Lebesgue measure, the Rademacher theorem is an obstruction to prescribe non-linear local behaviors at many points. Since by \cite{Za} we know that for a singular measure $\mu$ one can find Lipschitz functions which are $\mu$-almost everywhere non-differentiable, we wonder if it is possible to find a a Lipschitz function with an arbitrarily prescribed ``type'' of non-differentiability. Given a pair $(a,b)$ in $\R^2$ we say that a function $f:\R\to\R$ is {\emph{$(a,b)$-differentiable}} at the point $x_0$ if the two limits
$$\lim_{x\to x_0^-}\frac{f(x)-f(x_0)}{x-x_0}, \lim_{x\to x_0^+}\frac{f(x)-f(x_0)}{x-x_0}$$ 
exist and they are equal to $a$ and $b$ respectively. Note that, for a Lipschitz function, this is the only admissible type of ``first order approximation'' of $f$ at $x_0$ (i.e. the only possible behavior which is invariant under blowups). The following proposition shows that in general, even if $\mu$ is singular, it is not possible to prescribe a non-linear first order approximation at many points.
\begin{proposition}\label{grad2}
Let $\mu$ be a Radon measure on $\R$ and let $a, b:\R\to\R$ be bounded Borel functions, such that $a(x)\neq b(x)$ $\mu$-a.e. Then the following property $(P)$ holds if and only if $\mu$ is an atomic measure. 

$(P)$ For every $\delta>0$ there exist a set $E\subset\R$ with $\mu(\R\setminus E)<\delta$ and a Lipschitz function $f:\R\to\R$ such that for every $x\in E$, $f$ is $(a(x),b(x))$-differentiable.
\end{proposition}
\begin{proof}
If $\mu$ is an \emph{atomic measure} (i.e. there exists a countable set $N$ such that $\mu(\R\setminus N)=0$), then it is very easy to prove the validity of property $(P)$ constructing, for every $\delta$, an appropriate piecewise affine function $f$.\\ 
\noindent Now assume property $(P)$ holds for the functions $a$ and $b$. Then it also holds if we replace $a$ and $b$ with 
$$a_1:=a-(a+b)/2,\;\;\; b_1:=b-(a+b)/2.$$
Indeed, given $\delta>0$ one can apply Theorem \ref{grad1} to the function $g:=(a+b)/2$ with parameter $\varepsilon:=\delta/2$, thus obtaining a set $E_0$ and a function $f_0$. Then it is sufficient to find a set $E_1$ and a function $f_1$ satisfying property $(P)$ for $a$, $b$ and $\delta/2$. Hence the function $f:=f_0+f_1$ and the set $E:=E_0\cap E_1$ yield property $(P)$ for the fixed parameter $\delta$ and the functions $a_1$ and $b_1$.\\
\noindent Now we have that $a_1(x)$ and $b_1(x)$ have different sign (non zero) for $\mu$-almost every point $x$. Note that this implies that $\mu$-almost every $x\in E$ is a strict local maximum or minimum for $f$. We claim that there are at most countably many such points, which implies that $\mu$ is an atomic measure. To prove the claim, for every $i\in\N$ we denote by $A_i$ the set of points $x$ in $E$ such that $f(x)$ is the unique minimum of $f$ in the interval $(x-1/i,x+1/i)$. By construction, the set $A_i$ is discrete for every $i\in\N$, hence the union of the sets $A_i$ (which contains $\mu$-a.e. point of $E$) is at most countable.
\end{proof}

\noindent Even if for a general measure it is not possible to prescribe any form of non-differentiable first order approximation, it might be possible to prescribe the existence of a non-linear blow up, at many points $x$. By this we mean to find for those points $x$ a sequence of radii such that the corresponding rescaled functions converge to the prescribed non-linear blow-up at the point. Note that both the prescribed blow-up and the scales at which it is attained may vary from point to point. Such problem is treated in \cite{MS}.

\bibliographystyle{plain}

%
%

\vskip .5 cm

{\parindent = 0 pt\begin{footnotesize}

A.M.
\\
Institut f\"ur Mathematik,
Mathematisch-naturwissenschaftliche Fakult\"at,
Universit\"at Z\"urich\\
Winterthurerstrasse 190,
CH-8057 Z\"urich,
Switzerland
\\
e-mail: {\tt andrea.marchese@math.uzh.ch}

\end{footnotesize}
}

\end{document}